\documentclass[twocolumn]{autart}    

\usepackage{grffile}
\usepackage{amsmath,amssymb,amsbsy}
\usepackage{hyperref}
\usepackage{color}
\usepackage{comment}
\usepackage{lipsum}
\usepackage{graphicx}
\usepackage[normalem]{ulem}
\usepackage{enumerate}
\usepackage{epstopdf}
\usepackage{bbm}
\usepackage[caption=false,font=footnotesize]{subfig}
\newtheorem{theorem}{Theorem}[section]
\newtheorem{lemma}[theorem]{Lemma}
	
\newtheorem{proposition}[theorem]{Proposition}

\newtheorem{assumption}[theorem]{Assumption}
\newenvironment{proof}{\textbf{Proof:}}{\hfill$\square$}

\newcommand{\y}{\mathbf{y}}

\newcommand{\Pm}{\mathcal{P}}
\newcommand{\R}{\mathbb{R}}
\newcommand{\Rd}{\mathbb{R}^d}
\newcommand{\tc}{\tilde{c}}
\newcommand{\tT}{\tilde{T}}
\newcommand{\tx}{\tilde{x}}
\newcommand{\ty}{\tilde{y}}

\begin{document}

\begin{frontmatter}

\title{Optimal Transport of Linear Systems over Equilibrium Measures} 

\thanks[footnoteinfo]{This paper was not presented at any IFAC 
meeting.}

\author[a]{Karthik Elamvazhuthi}\ead{karthike@ucr.edu},    
\author[b]{Matt Jacobs}\ead{majaco@ucsb.edu}               

\address[a]{Department of Mechanical Engineering, University of California, Riverside}  
\address[b]{Department of Mathematics, University of California, Santa Barbara}             

\begin{keyword}                           
Optimal Transport; Linear Time Invariant Systems; Distribution Control.               
\end{keyword}                             

\begin{abstract}                          
We consider the optimal transport problem over convex costs arising from optimal control of linear time-invariant(LTI) systems when the initial and target measures are assumed to be supported on the set of equilibrium points of the LTI system. In this case, the probability measures are singular with respect to the Lebesgue measure, thus not considered in previous results on optimal transport of linear systems. This problem is motivated by applications, such as robotics, where the initial and target configurations of robots, represented by measures, are in equilibrium or stationary. Despite the singular nature of the measures, for many cases of practical interest, we show that the Monge problem has a solution by applying classical optimal transport results. Moreover, the problem is computationally tractable even if the state space of the LTI system is moderately high in dimension, provided the equilibrium set lives in a low dimensional space. In fact, for an important subclass of linear quadratic problems, such as control of the double integrator with linear quadratic cost, the optimal transport map happens to coincide with that of the Euclidean cost. We demonstrate our results by computing the optimal transport map for the minimum energy cost for a two dimensional double integrator, despite the fact that the state space is four dimensional due to position and velocity variables. 
\end{abstract}

\end{frontmatter}

\section{Introduction}

The optimal transportion problem is the study of transferring resources from one configuration to another, where resources are modeled as probability distributions. Motivated initially by problems in economics, it has since been a subject of intense research in the mathematics literature \cite{villani2009optimal,santambrogio2015optimal} and has found applications in a wide range of areas such as computer vision \cite{rubner2000earth}, signal processing \cite{kolouri2017optimal} and multi-robot control \cite{kabir2021efficient}. Classical formulations of optimal transport problems consider situations where the cost of the particle is defined explicitly between any two points in space. However, this excludes a large class of costs arising from control theory where the cost itself is a solution to an optimal control problem. Towards this end, there has been some recent effort on considering optimal transport problems for optimal control costs \cite{agrachev2009optimal,hindawi2011mass,chen2016optimal,de2021discrete,elamvazhuthi2018optimal,elamvazhuthi2023dynamical}. See also related work on control of densities of stochastic systems \cite{karny1996towards,fleig2017optimal,pakniyat2022convex}, Schrodinger bridges \cite{ito2023maximum,chen2016relation}, covariance steering of control systems \cite{okamoto2018optimal,balci2023covariance} and stabilization of probability measures \cite{biswal2021decentralized}.

Existing works on optimal transport of control systems have focused on the case when the source and target measures are absolutely continuous with respect to the Lebesgue measure, where it has been shown that techniques from optimal transport theory can be combined with optimal control principles to construct transport maps. However, there exists a distinctive class of problems within this domain that has garnered less attention but holds significance in various practical scenarios – those involving probability measures supported on the set of equilibrium points.

This paper introduces the optimal transport problem over convex costs for linear time invariant control systems with initial and target measures assumed to be supported exclusively on equilibrium points. Unlike typical scenarios, this particular problem deals with measures that are singular with respect to the Lebesgue measure. This singular nature of the measures has, until now, been largely unexplored within the context of optimal transport for linear systems.

The problem that we consider is motivated by practical applications in fields such as robotics, where often swarms of robots are required to be transported from one equilibrium configuration to another. This implies that the proportion of robots in a state of zero velocity in initial and final configuration, is always one. Despite the singular nature of the measures involved, we show that, for many cases of practical interest, solutions to the {\it Monge problem}, that is the existence of deterministic maps, can be obtained by leveraging classical optimal transport results. Furthermore, the paper demonstrates that not only can solutions be shown to exist, but they are also computationally tractable, even when dealing with moderately high-dimensional state spaces. The key lies in the dimensionality of the equilibrium set, which, when residing in a low-dimensional subspace, allows for efficient computation of the optimal transport map.

To emphasize the practical relevance of these findings, the paper examines a specific subclass of linear quadratic problems, notably a class of control systems decoupled along different dimensions of the state space. Remarkably, in this particular case, the optimal transport map aligns with that of the Euclidean cost on the equilibrium space. A special case of this example is the $N-$dimensional double integrator. For the $2$-D case we compute the transport map using the Back and Forth method \cite{jacobs2020fast}, a fast method for optimal transport that can compute deterministic maps, and visualize the trajectory of the probability distribution by interpolating the map using optimal control trajectories.

The paper is organized as follows. In Section \ref{sec:bkg}, we provide a brief background and motivation behind the problem. In addition, we present the problem that we analyze in this paper. In Section \ref{sec:asys}, we state and prove our main result (Theorem \ref{thm:mr}) on the existence of the solutions to the Monge formulation of the optimal transport problem for a large class of convex costs. In Section \ref{sec:eg}, we present some example optimal transport problems with LQ costs that can be computed using existing optimal transport algorithms, even if the state-space is of higher dimension than 3.

\section{Background and Motivation}
\label{sec:bkg}
In this section, we give a brief background and motivation behind the problem. 

Let $\mathcal{P}_2(\mathbb{R}^d)$ denote the set of Borel probability measures on $\mathbb{R}^d$ with finite second moment: $\int_{\Omega} |x|^2d\mu(x)~<~\infty$. For a given Borel map $T : \mathbb{R}^d \rightarrow  \mathbb{R}^d$ we will denote by $T_{\#}$ the corresponding pushforward map, which maps any measure $\mu $ to a measure $T_{\#}\mu$, where $T_{\#}\mu$ is the measure defined by
	\begin{equation}
	(T_\# \mu)(\Omega) = \mu(T^{-1}(\Omega)), 
	\end{equation}
	for all Borel measurable sets $\Omega \subseteq \mathbb{R}^d$. 
 The support of a measure $\mu \in \mathcal{P}(\mathbb{R}^d)$ will be denoted by ${\rm supp} ~ \mu = \lbrace \mathbf{x} \in \mathbb{R}^d; ~\mathbf{x} \in N_\mathbf{x} ~  {\rm implies~that} ~ \mu(N_\mathbf{x}) >0, ~{\rm where~} N_\mathbf{x} ~ {\rm  is~a~ neighborhood ~ of~} \mathbf{x}  \rbrace$. 
 
The optimal transport problem is the following
 \begin{eqnarray}
\label{eq:OTprb}
\inf_{T} \int_{x \in \mathbb{R}^d} c(x,T(x))dx \\
s.t. ~~ T_{\#} \mu = \nu
\end{eqnarray}
This problem has been well studied in the literature \cite{villani2009optimal,santambrogio2015optimal}. In an engineering context, $\mu, \nu$ could represent the position of swarm of robots or network of power systems. The well-posedeness of the problem when the costs arise form optimal control problems when the associated control system is linear \cite{hindawi2011mass,chen2016optimal} and system is nonlinear \cite{agrachev2009optimal,elamvazhuthi2023dynamical}. 

In this paper, we will be considering the case when the control system is a linear time invariant (LTI) system. Let $A\in \mathbb{R}^{d \times d}$, $B \in \mathbb{R}^{d \times m}$. The running cost is denoted by the function $ L: \mathbb{R}^d \times \mathbb{R}^m \rightarrow \mathbb{R} $. Using these, we define the cost function $c: \mathbb{R}^d \times \mathbb{R}^d \rightarrow \mathbb{R}$ given by
\begin{eqnarray}
\label{eq:occost}
c(x,y) = \inf_{(\gamma,u) \in \Gamma \times \mathcal{U} } \int_0^1 L(\gamma(t),u(t))dt  \\
\text{s.t.} ~~~\frac{d\gamma}{dt} = A\gamma(t) +Bu(t); \label{eq:lti}\\
\gamma(0)=x,~~\gamma(1) = y \label{eq:edp}.
\end{eqnarray}
where $\Gamma = AC([0,T];\R^d) $ is the set of absolutely continuous curves on $\R^d$ and $\mathcal{U} = L^2(0,1;\mathbb{R}^d)$ is the set of square integrable control inputs. We assume that the system is controllable, so that the optimization constraints with end points $\gamma(0)=x~~\gamma(1) = y$ is feasible for every $x,y \in \mathbb{R}^d$.

An example of this instance is the case when, the cost is given by the minimum energy control for the \textbf{2-dimensional double integrator}, where the control is the acceleration of a particle moving in $2$-dimensional Euclidean space.
\begin{eqnarray}
c(x,y) = \inf_{(\gamma,u) \in \Gamma \times \mathcal{U} } \int_0^1 |u(t)|^2dt  \\
\text{s.t.} ~~~\frac{d\gamma}{dt} = 
\gamma(t)
\begin{bmatrix}
0 & 1 & 0 & 0 \\
0 & 0 & 0 & 0 \\
0 & 0 & 0 & 1 \\
0 & 0 & 0 & 0 \\
\end{bmatrix}
+\begin{bmatrix}
0 & 0 \\
1 & 0 \\
0 & 0 \\
0 & 1 
\end{bmatrix} u(t)
\\
\gamma(0)=x,~~\gamma(1) = y 
\end{eqnarray}
where $\gamma(t) \in \mathbb{R}^4 $ and $u(t) \in \mathbb{R}^2$  for all $t$.
The optimal transport map for the above cost is challenging to compute for general probability measures since the state space of the system is four dimensional. The numerical solvability of the problem requires that $\mu, \nu$ is absolutely are absolutely continuous with respect to the Lebesgue measure on $\mathbb{R}^4$. 

On the other hand, in many practical situations, such as in robotics, it is often the case that particles are at state of equilibrium in both initial and final configurations. For numerical solvability and practical considerations of transporting particles from one equilibrium configuration to another, we consider a alternative set of probability measures. We define the set of equilibrium points, 
\[E : = \{\gamma \in \mathbb{R}^d ;A \gamma =0\} \big \}. \]
In the case of the double integrator above, this set is the set of all zero velocity states
\[E_{di} : = \{x\in \mathbb{R}^4 ; \exists p_1, p_2 \in \mathbb{R}; x = [p_1~0~p_2~0]\} \big \}, \]
Using the set $E$ we define the the the set of \textbf{equilibrium measures},  $\mathcal{P}^{eq}_2(\mathbb{R}^d)$, given by
\[\mathcal{P}^{eq}_2(\mathbb{R}^d) = \big \{\mu \in \mathcal{P}_2(\mathbb{R}^d); {\rm supp} ~\mu \subseteq E \} \]
We will consider the optimal transport problem for measures $\mu, \nu \in \mathcal{P}^{eq}_2(\mathbb{R}^d)$, thus potentially reducing the dimension of the problem. Note that since the measure $\mathcal{P}^{eq}_2(\mathbb{R}^d)$ are non-singular, the results of \cite{hindawi2011mass,chen2016optimal} do not apply. \textbf{The goal of the paper will be to provide sufficient conditions for the existence of the optimal transport map when the source and target measure are equilibrium measures}. 

Before we begin our analysis we state some assumptions on the cost function $L$. Toward this end, let $P_E :\mathbb{R}^d \rightarrow \mathbb{R}^d$ be the projection map from $\mathbb{R}^d$ to $\mathbb{R}^d$, onto the set $E$. Let $P_E^{\perp}$ be the projection map onto the orthogonal complement of $E$.

\begin{assumption}
\label{asmp:asmp1}
The cost function $L:\mathbb{R}^d \times \mathbb{R}^d \rightarrow \mathbb{R}^d$ satisfies the following properties
\begin{enumerate}
\item There exists $\theta$ such that $L(\gamma,\cdot) \geq \theta |u|^2$ for all $ (\gamma,u) \in \mathbb{R}^d \times \mathbb{R}^m$
\item $L $ is a convex function  that is uniformly strongly convex with respect to the control variable. That is, there exists $\lambda>0$ such that 
\begin{eqnarray}
\ & L(\alpha \gamma +(1-\alpha)\eta,\alpha u + (1-\alpha)v)) \leq  \nonumber
\\ & \alpha L(\gamma,u) +(1-\alpha )L(\eta,v) -\frac{\alpha(1-\alpha)\lambda}{2}\|u-v\|^2_2 \nonumber
\end{eqnarray}
for all $\gamma,\eta \in \mathbb{R}^d$, all $u,v \in \mathbb{R}^m$ and all $\alpha \in (0,1)$
\item $L(\gamma,u) = L(P_{E_{\perp}}\gamma,u)$ for all $(\gamma,u) \in \mathbb{R}^d \times \mathbb{R}^m$
\end{enumerate}
\end{assumption}

The first three assumption are standard assumptions on regularity and convexity of the cost functional \cite{fleming2012deterministic}. Assumption 4 implies that the cost acts on the state, only through the non-equilibrium variables. For example, in the case of the double integrator, this implies that the state dependence of the cost is only through the velocity variables.

\section{Analysis}
\label{sec:asys}
In section, we present the analysis for the Monge problem presented in the previous section. First, we note the classical result on the existence of solutions of optimal transport maps. The strategy of the paper will be to verify that the optimal control cost on the equilibrium set satisfies all these assumptions. 
\begin{theorem} \cite{santambrogio2015optimal}[Theorem 1.17]
\label{thm:clot}
Let $\tilde{c}: \mathbb{R}^d \times \mathbb{R}^d \rightarrow \mathbb{R}$ be such that $\tilde{c}(x,y) = \tilde{h}(x-y)$ with $\tilde{h} :\mathbb{R}^d \rightarrow \mathbb{R}$ strictly convex. Assume $\mu, \nu \in \mathcal{P}_2(\Omega)$ have compact supports and are absolutely continuous with respect to the Lebesgue measure. Then there exists a unique solution to the optimal transport problem.  
\begin{eqnarray}
\inf_{T} \int_{x \in \mathbb{R}^d} \tilde{c}(x,T(x))dx \\
s.t. ~~ T_{\#} \mu = \nu
\end{eqnarray}
\end{theorem}

An additional result that we will leverage from the optimal control literature is the following one on the existence of solutions to the optimal control problem. This will ensure that the point to point optimal control problem \eqref{eq:occost} is well-posed, which is a prerequisite to the well-posedness of the optimal transport problem \eqref{eq:OTprb}.

\begin{theorem}
\cite{fleming2012deterministic}[Chapter III, Corollary 4.1]
Given Assumption \ref{asmp:asmp1}, for each $x,y \in \mathbb{R}^d$, the cost function is well defined and there exists an optimal pair $(\gamma, u) \in \Gamma \times \mathcal{U}$ that achieves the infimum.
\end{theorem}

Next we note a result that if one starts and ends at an equilibrium set, the negative of the control achieves the same state transition in the opposite direction, from $y$ to $x$.
\begin{lemma}
Suppose $x,y \in E$ and $u \in \mathcal{U}$ is such that 
\begin{eqnarray}
\frac{d\gamma}{dt} = A\gamma(t) +Bu(t); \\
\gamma(0)=x,~~\gamma(1) = y .
\end{eqnarray}
then,
\begin{eqnarray}
\label{eq:rev}
\frac{d\tilde{\gamma}}{dt} = A\tilde{\gamma}(t) -Bu(t); \\
\tilde{\gamma}(0)=x,~~\tilde{\gamma}(1) = y .
\end{eqnarray}
where $\tilde{\gamma} (t) =-\gamma(t)+x+y$.
\end{lemma}
\begin{proof}
The equality is easy to see by plugging in the expression  $\tilde{\gamma} (t) =-\gamma(t)+x+y$ in equation \eqref{eq:rev},
\end{proof}

Next, we establish the convexity of the cost function $c$. Moreover, the previous Lemma will help us to show that, on the equilibrium set, the cost is dependent only on the difference between the initial and the target point.

\begin{proposition}
\label{prop:cvx}
Given Assumption \ref{asmp:asmp1}, there exists $h:E \rightarrow \mathbb{R}$ such that $c(x,y) = h(x-y)$ for all $x ,y \in E$.

Moreover, $h$ is strictly convex. 
\end{proposition}
\begin{proof}
First we establish the existence of the function $h$. Let $x_1,y_1,x_2,y_2 \in E$ be such that $x_1-y_1 = x_2-y_2$. Suppose $(\gamma^o, u^o) \in \mathcal{U}$ is the optimal solution for the Problem \ref{eq:occost} with $x= x_1$ and $y=y_1$. Then consider the trajectory $\gamma :[0,1] \rightarrow \mathbb{R}^d$, given by $\gamma(t) = \gamma^o(t) -x_1 +x_2$ for almost every $t \in [0,1]$.
\begin{align*}
\int_0^1 L(\gamma^o(t), u^o(t))dt & = \int_0^1 L(P^{\perp}_E\gamma(t), u^o(t))dt  \\
&= \int_0^1 L(\gamma(t), u^o(t))dt  
\end{align*}
Moreover, the pair $\gamma \in \Gamma$ is a solution of the linear system  \eqref{eq:lti}-\eqref{eq:edp} for control $u =u^o$ and $\gamma(0) = x_2, \gamma(1)=y_2$ since,
\begin{eqnarray}
 \frac{d\gamma}{dt} =\frac{d\gamma^o}{dt} & =A\gamma^o(t) +Bu^o(t) \\
 &= A\gamma(t) +Bu^o(t) ; 
\end{eqnarray}
This shows that $c(x,y)$ only depends on $x,y$ and hence, there exists a function $h:E \rightarrow \R^d$ such that $c(x,y) = h(x-y)$. 

Next, we prove that $h$ is strictly convex on $E$. Let $x_0,y_0,x_1,y_1 \in E$. Define $z_0 = x_0-y_0$ and $z_1 = x_1-y_1$. Assume that $z_0 \neq z_1$. Let $(\gamma_0,u_0), (\gamma_1,u_1) \in \Gamma \times \mathcal{U}$ be the optimal solution for the problem \eqref{eq:occost}-\eqref{eq:edp}, for $(x,y) = (x_0,y_0)$ and $(x,y) = (x_1,y_1)$, respectively. Next, we define the sub-optimal convex combination of $(\gamma_1,u_1)$ and $(\gamma_2,u_2)$ by
\[(\gamma_\alpha,u_\alpha) = ((1-\alpha) \gamma_0+\alpha \gamma_1,(1-\alpha) u_0+\alpha u_1)  \]
Note that $u_\alpha$ is a control corresponding to the trajectory $\gamma_\alpha$.
From the assumption on the convexity of $L$, we can conclude that,
\begin{align}
\label{eq:cvtest}
 L(\gamma_\alpha (t), u_\alpha (t) ) 
\leq &  \alpha L(\gamma_1 (t), u_1 (t) )\nonumber  \\ 
& + (1-\alpha) L(\gamma_2 (t), u_2 (t) ) \nonumber \\
& + \frac{\alpha(1-\alpha)\alpha}{2}\|u_1(t)-u_2(t)\|^2_2 
\end{align}
for almost every $t \in [0,1]$ and all $\alpha \in (0,1)$.
Since $z_1 \neq z_2$, $u_1 \neq u_2$.  Otherwise, from the solution of the linear system we would get that,
\begin{eqnarray}
\gamma_n(1) &= e^{A}x_n + \int_0^1Be^{A(1-\tau)}u_n(\tau)d\tau   \\
& = x_n  +\int_0^1Be^{A(1-\tau)}u_n(\tau)d\tau  
\end{eqnarray}
for $n=0,1$. Since, $\gamma_n(1) =y_n$, we would have that
\begin{equation}
y_n -x_n=\int_0^1Be^{A(1-\tau)}u_n(\tau)d\tau 
\end{equation}
for $n=0,1$. Thus, contradicting with the assumption that $z_0 \neq z_1$.

Therefore, from \eqref{eq:cvtest} we can conclude that
\begin{eqnarray}
& &\int_0^1L(\gamma_\alpha (t), u_\alpha (t) )dt  \nonumber  \\ 
<& & \int_0^1 [\alpha L(\gamma_1 (t), u_1 (t) ) + (1- \alpha) L(\gamma_2 (t), u_2 (t) )]dt 
\end{eqnarray}
for all $\alpha \in (0,1)$.
 Let $(\hat{\gamma}_{\alpha}, \hat{u}_\alpha) \in \Gamma \times \mathcal{U} $ be the optimal solution of problem \eqref{eq:occost}-\eqref{eq:edp} for  $(x,y) = ((1-\alpha) x_0+\alpha x_1,(1-\alpha) y_0+\alpha y_1)$.
This implies,
\begin{eqnarray}
h(\alpha z_0 +(1-\alpha)z_1) &=& c(\alpha x_0 +(1-\alpha)x_1,\alpha x_0 +(1-\alpha)x_1) \nonumber \\ &=&   \int_0^1L(\hat{\gamma}_\alpha (t), \hat{u}_\alpha (t) )dt \nonumber \\ &  <&  \int_0^1L(\gamma_\alpha (t), u_\alpha (t) )dt  \nonumber \\ &=& \alpha h(z_0) +(1-\alpha)h(z_1)\nonumber
\end{eqnarray} 
This concludes the proof.

\end{proof}

Now we are ready to state our main on existence of solution to the Monge problem.
\begin{theorem}
\label{thm:mr}
Let $R:E \rightarrow \R^p$ be an invertible linear map, where $p$ is the dimension of the set $E$. Given Assumption \ref{asmp:asmp1}, and $\mu,\nu \in \mathcal{P}^{eq}_2(\Rd)$ are such that $R_{\#} \mu, R_{\#} \nu \in \Pm(\R^p)$ are absolutely continuous with respect to the Lebesgue measure on $\R^p$ and have compact supports. Then there exists a unique solution $T$ to the optimal transport problem.  
 \end{theorem}
\begin{proof}
Given a cost function $c$, we can define a cost $\tc(\tx,\ty)$ function on $\R^p$ by $\tc(\tx,\ty) = c(R^{-1}\tx, R^{-1}\ty)$ for all $(\tx,\ty)$. Then solving the optimal transport problem \eqref{eq:occost} is equivalent to solving the optimal transport problem 
\begin{eqnarray}
\inf_{T} \int_{x \in \mathbb{R}^d} \tilde{c}(x,T(x))dx \\
s.t. ~~ T_{\#} \mu = \nu
\end{eqnarray}
Given a solution $\tT^*_o$ to the above problem it can be verified that $T^*_o = R^{-1}\tT^*_o(Rx)$ for $\tilde{\mu} = R_{\#}\mu$ and $\tilde{\nu} = R_{\#}\nu$ is the optimal solution for the optimal transport problem. Using the definition of the pushforward map, we have that
\begin{align*}
\int_{\R^d}c(x,T_o(x))d\mu(x) & =\int_{\R^d}c(R^{-1}\tx,T_o(R^{-1}\tx))d\tilde{\mu}(\tx)   \\ 
  \\
&=\int_{\R^d}\tc(\tx,RT_o(R^{-1}\tx))d\tilde{\mu}(\tx)\\
& =\int_{\R^d}\tc(\tx,\tT_o(\tx))d\tilde{\mu}(\tx)
\end{align*}
    The result follows from the above observation and a direct application of Proposition \ref{prop:cvx} and Theorem \ref{thm:clot}.
\end{proof}

\section{Examples}
\label{sec:eg}
In this section, we consider important example cases where one can compute the transport map using existing numerical methods, despite the fact that the state space dimension of the system is greater than $3$.

\textbf{Minimum Energy Costs} Let $A \in \mathbb{R}^{d \times d}$ be a matrix such that the equilibrium set is given by $E =  \{ x \in \mathbb{R}^d ; \exists c \in \mathbb{R}, ~ x = [0 ~...~ c]^T \}$.
Then we consider the cost $c :\mathbb{R} \times \mathbb{R} \rightarrow \mathbb{R}$ given by
\begin{eqnarray}
c(x,y) = \inf_{(\gamma,u) \in \Gamma \times \mathcal{U} }\int_0^1 u^2(t)dt 
\end{eqnarray}
subject to 
\begin{eqnarray}
& 
\dot{\gamma}
=A
\gamma(t)
+
B u(t) \\
&
\gamma(0)
 = \begin{bmatrix}
0 \\ \vdots \\ x
\end{bmatrix}
~~ 
\gamma(1)
= \begin{bmatrix}
0 \\  \vdots \\y
\end{bmatrix}
\end{eqnarray}
Define the controllability Grammian $W \in \mathbb{R}^{d \times d}$.
\[W = \int_0^1 e^{A(1-\tau)}BB^te^{A^T (1-\tau)}d\tau\]
where $\mathbf{x} = [0,...,x]^T$ and $\mathbf{y} = [0,...,y]^T$. 
The cost $c$ is given by 
\begin{eqnarray}
c(x,y) = \langle (\mathbf{y}-e^{At}\mathbf{x}), W^{-1} (\mathbf{y}-e^{At}\mathbf{x}))\rangle \\
 = \langle (\mathbf{y}-\mathbf{x}), W^{-1} (\mathbf{y}-\mathbf{x})\rangle 
\end{eqnarray}
since $\mathbf{x},\mathbf{y} \in E$. Due to the structure of $\mathbf{x}$ and $\mathbf{y}$, we can conclude that
\begin{equation}
\label{eq:gram}
c(x,y) = W^{-1}_{dd}(x-y)^2,
\end{equation}
where $ W^{-1}_{dd}$ denotes $(d,d)^{{\rm th}}$ element of the matrix $W^{-1}$.
Therefore, \textit{the cost $c$ is the same as the squared Euclidean distance on $\mathbb{R}$}. Hence, one can use existing numerical solvers for optimal transport over Euclidean costs, to find the Monge map.

\begin{figure}[b]
\label{fig:IniFin}
\centering
\subfloat[Initial Measure]{
\includegraphics[width=0.17\textwidth]{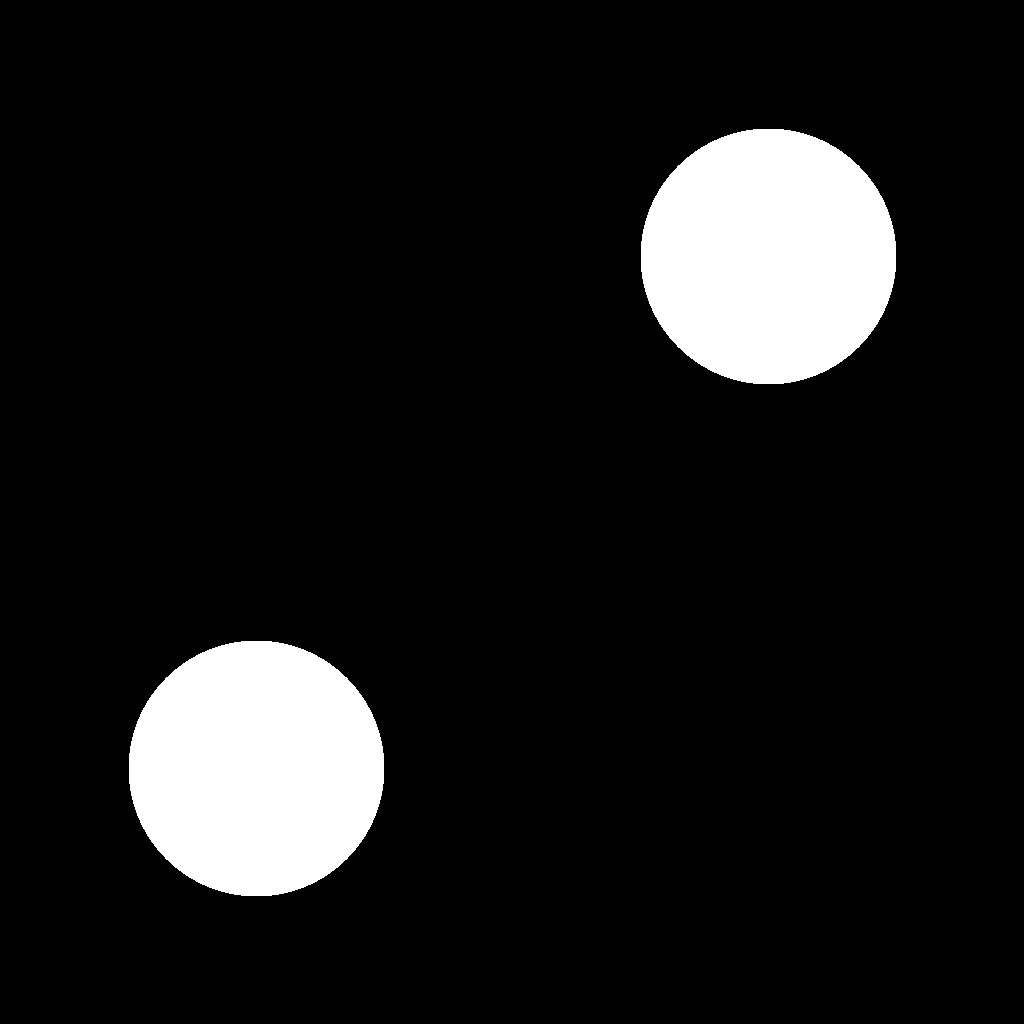}%
}%
\subfloat[Final Measure]{
\includegraphics[width=0.17\textwidth]{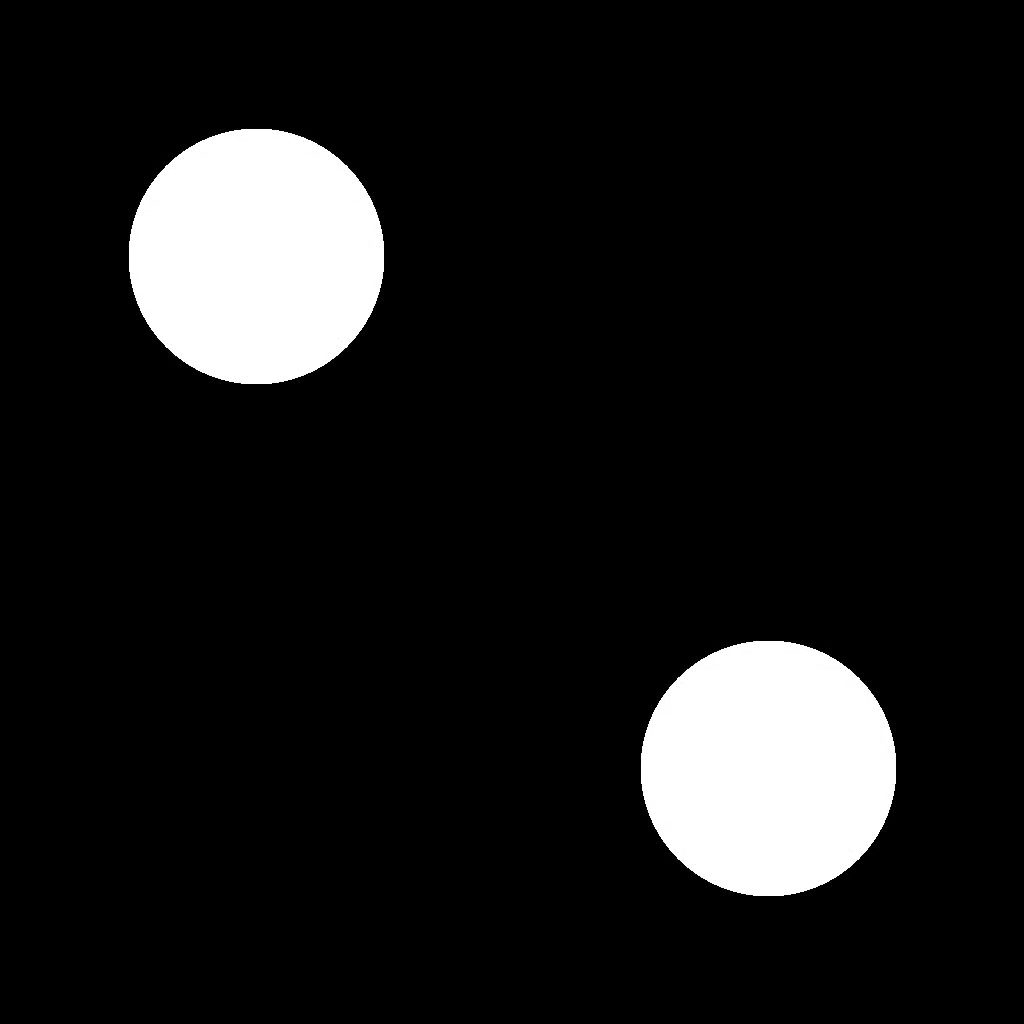}%
}%
\caption{Initial and Final measures for the transport problem.}
\end{figure}

An important special case is the \textbf{double integrator} system as discussed in Section \ref{sec:bkg}
Define 
\begin{equation}
A = \begin{bmatrix}
  0 & 1 \\
  0 & 0 
    \end{bmatrix}
    ~~~
    B =  \begin{bmatrix}
  0  \\
 1
    \end{bmatrix}.
\end{equation}
The cost induced on two dimensional position coordinates $\mathbb{R}^2$ is given  by
\[\tilde{c}(x,y) = \sum_{i=1}^nc(x_i,y_i)\]
In the case of the double integrator system, the position variable is evolving in this case in $n$-dimensional space, whereas the state-space of the system is $2n$- dimensional.
Given the expression \eqref{eq:gram}, $c$ defines the Euclidean distance on $\mathbb{R}^n$. Hence, one can compute the transport map using existing numerical methods for optimal transport. For the two dimensional case, we use the {\it Back and Forth method}, presented in \cite{jacobs2020fast} to compute the optimal transport map, and visualize the optimal trajectories of the particles, in position coordinates. While the cost itself is Eucliedean, the optimal trajectories are not constant velocity, unlike in the Euclidean case. The optimal control in this instance is known to be given by
\begin{equation}
\label{eq:optctrl}
u(t) = -B^*e^{A^*(1-t)}W^{-1}[e^{A}\mathbf{y}-\mathbf{x}]
\end{equation}

The initial measure $\mu$ is taken to be a measure uniformly distributed on two disks of radius $\frac{1}{8}$ centered at $(\frac{1}{4},\frac{3}{4})$ and $(\frac{3}{4},\frac{1}{4})$, respectively. The target measure $\nu$ is the measure uniformly distributed on two disks of radius $\frac{1}{8}$ centered at $(\frac{1}{4},\frac{1}{4})$ and $(\frac{3}{4},\frac{3}{4})$, respectively. The measures have been visualized in Figure \ref{fig:IniFin}.
We visualize the the evolution of the probability measures along optimal solutions of the double integrator computed using the expression in \eqref{eq:optctrl} in Figure \ref{fig:OT}. We have also shown the evolution of the measure along the trajectories of the Euclidean cost. One can see that, in the case of the double integrator, the particles have low initial and final velocity, but are faster than the case for the Euclidean trajectories at the intermediate time instances. Note that the Euclidean cost is the one considered in the classical optimal transport problem, and corresponds to the minimum energy cost single integrator system ($\dot{\gamma} = u$).

\begin{figure}[b]
\centering
\subfloat[$T=\frac{1}{5}$(Euclidean)]{
\includegraphics[width=0.17\textwidth]{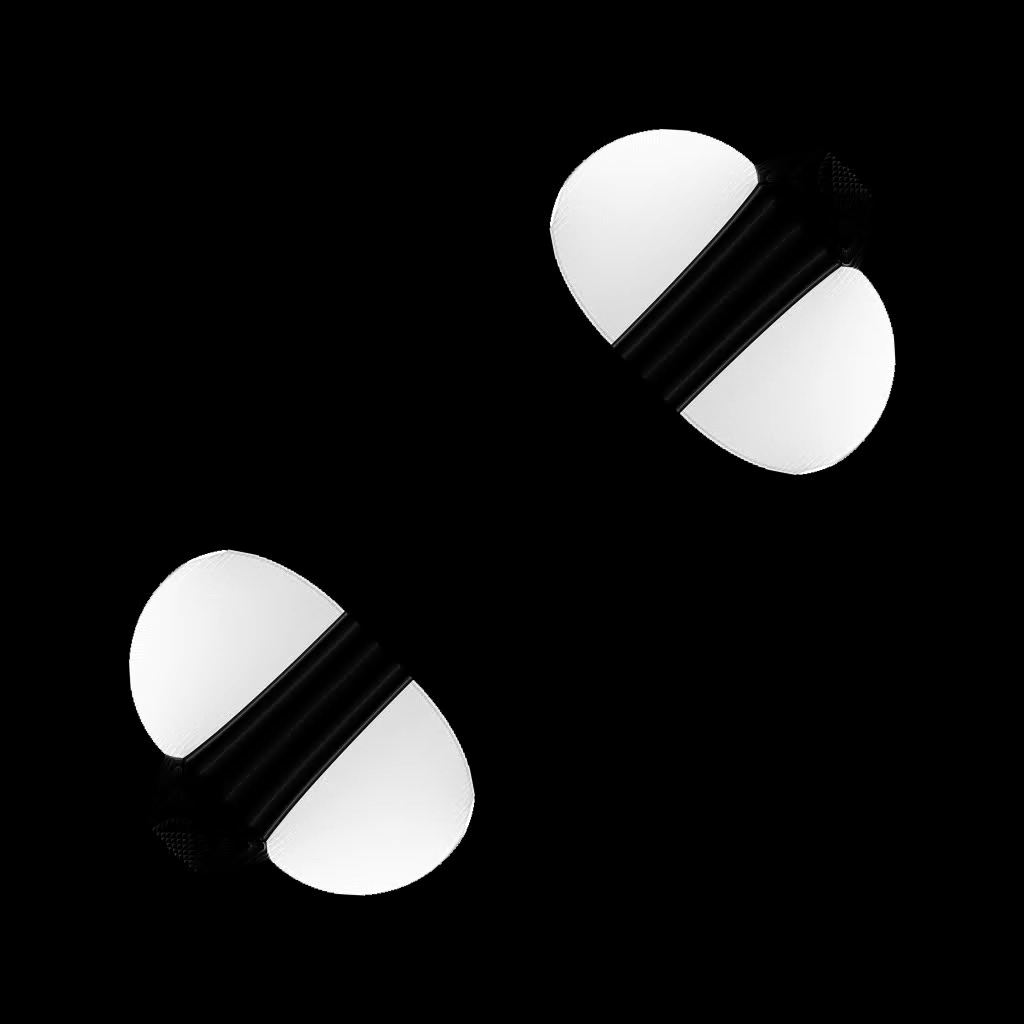}%
}%
\subfloat[$T=\frac{1}{5}$(DI)]{
\includegraphics[width=0.17\textwidth]{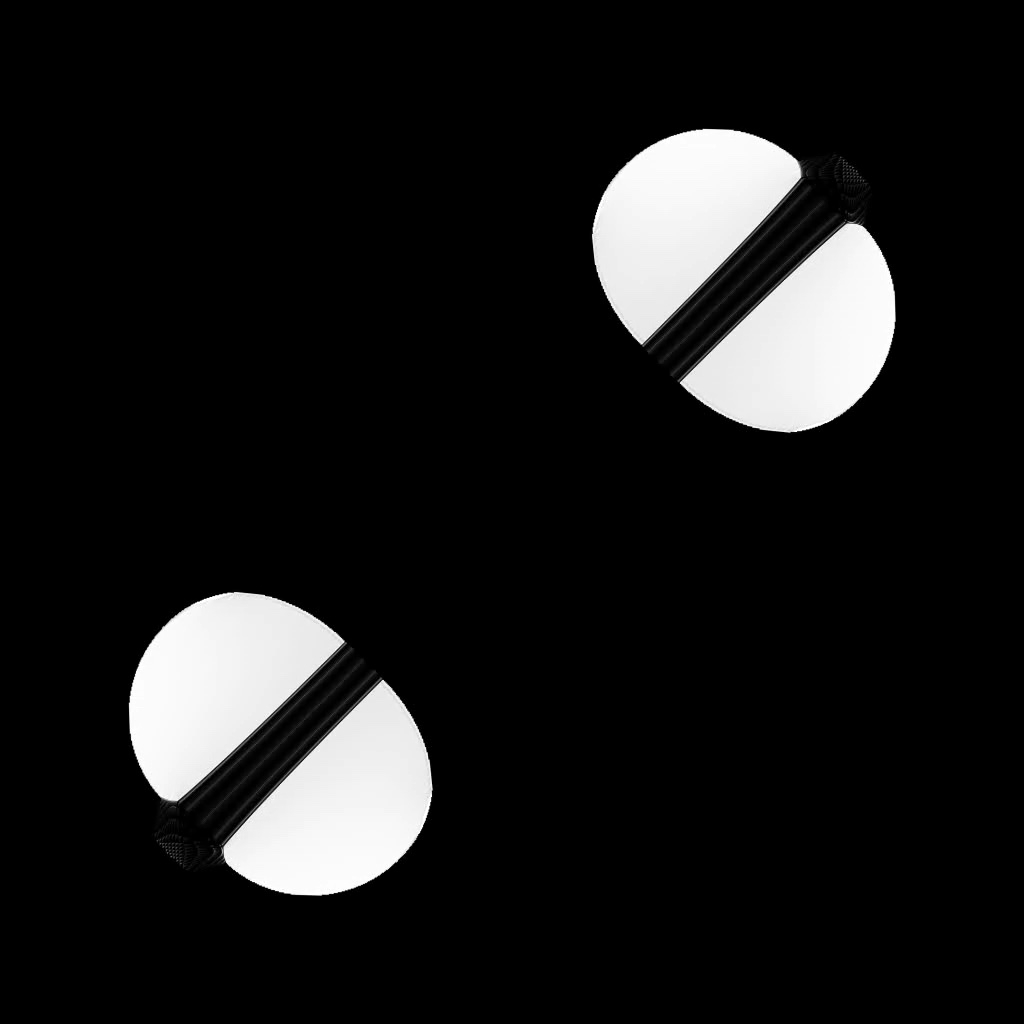}%
}%
\\
\subfloat[$T=\frac{2}{5}$(Euclidean)]{
\includegraphics[width=0.17\textwidth]{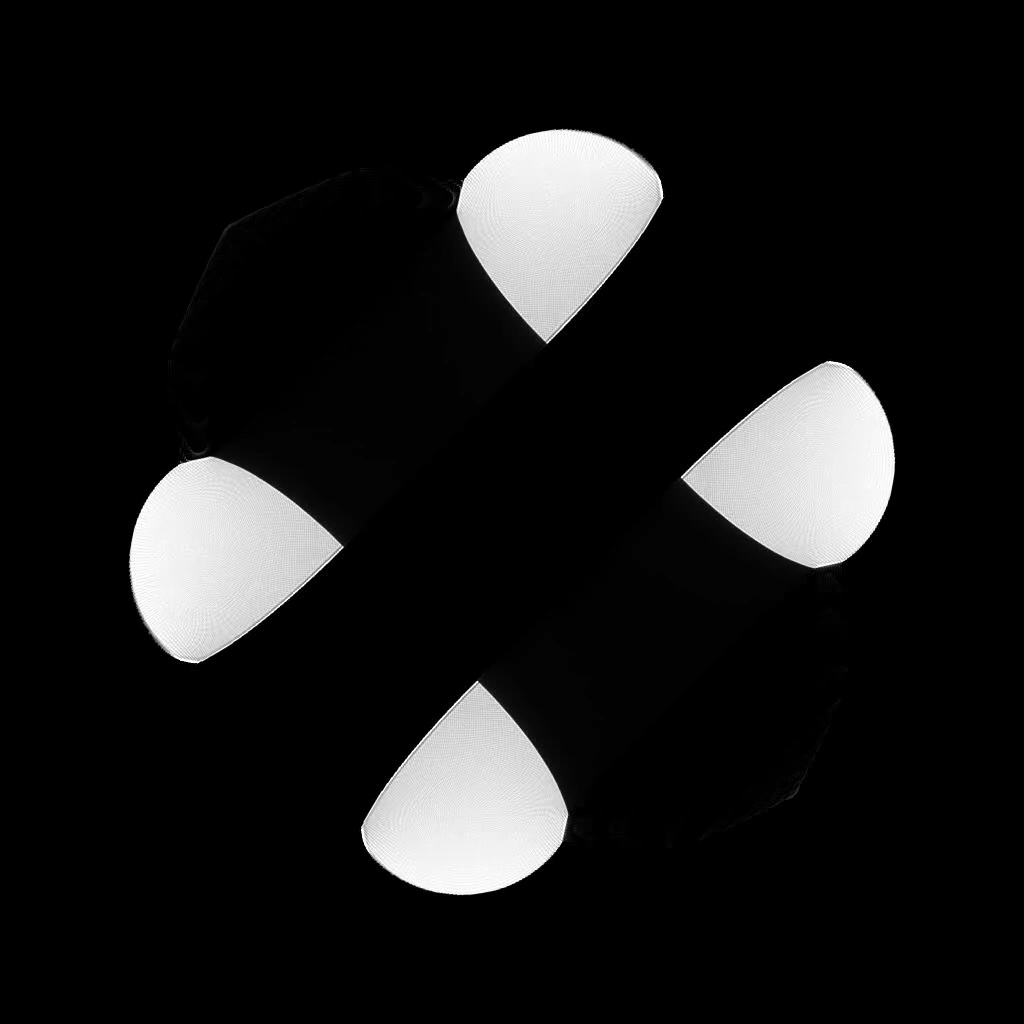}%
}%
\subfloat[$T=\frac{2}{5}$(DI)]{
\includegraphics[width=0.17\textwidth]{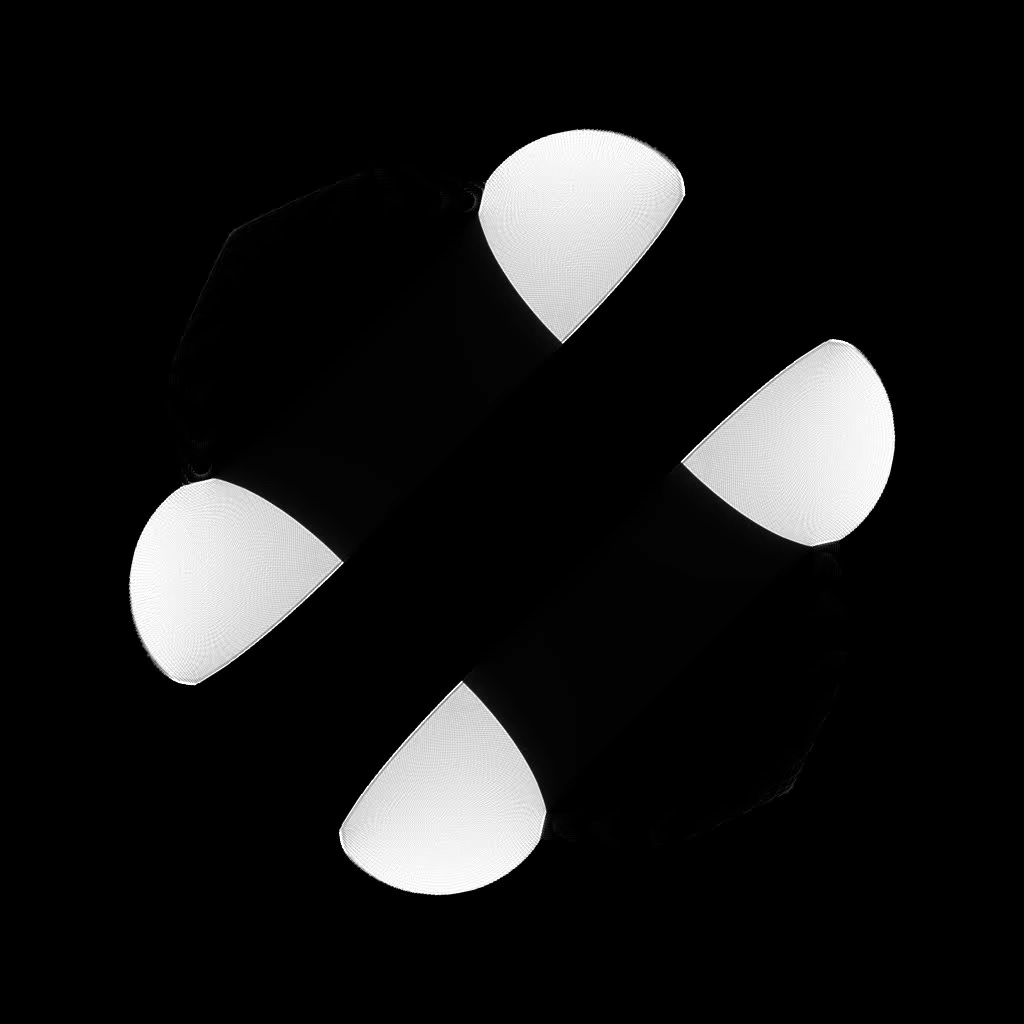}%
}%
\\
\subfloat[$T=\frac{3}{5}$(Euclidean)]{
\includegraphics[width=0.17\textwidth]{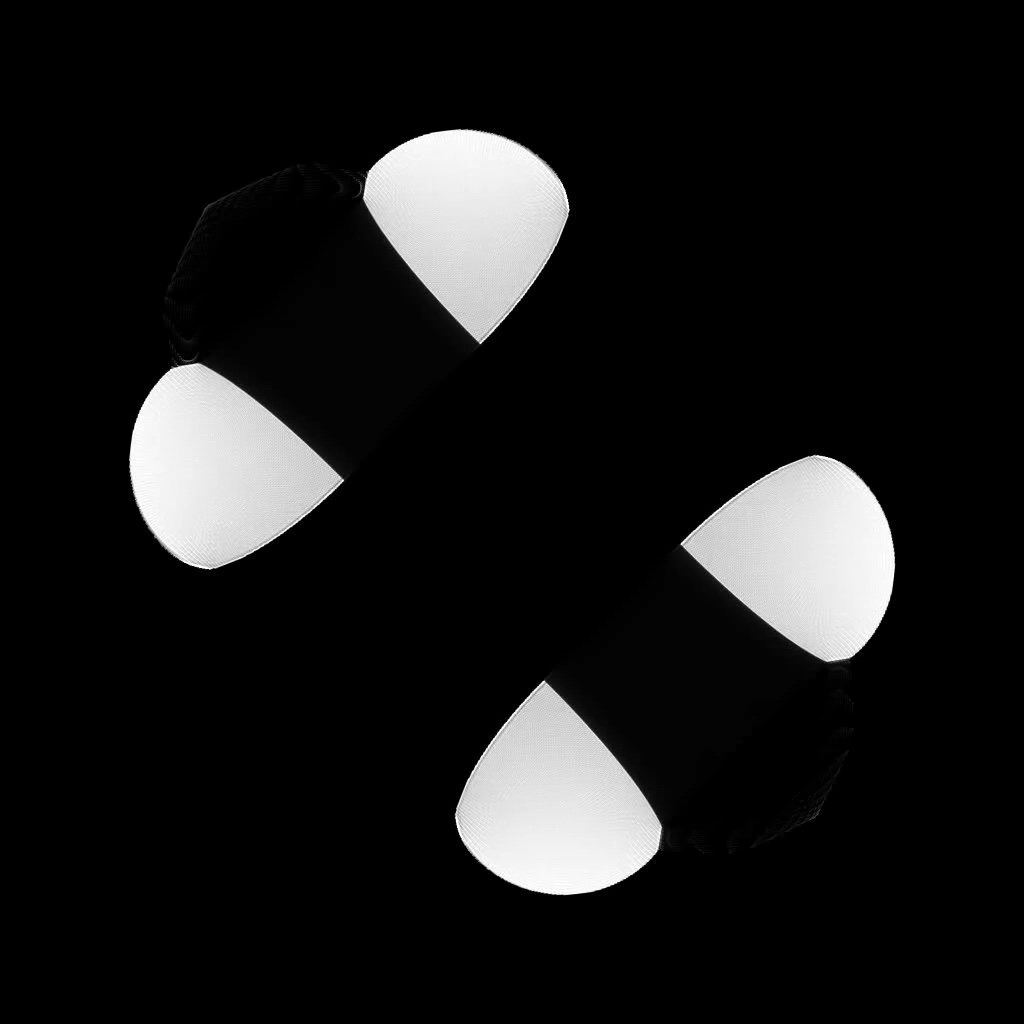}%
}
\subfloat[$T=\frac{3}{5}$(DI)]{
\includegraphics[width=0.17\textwidth]{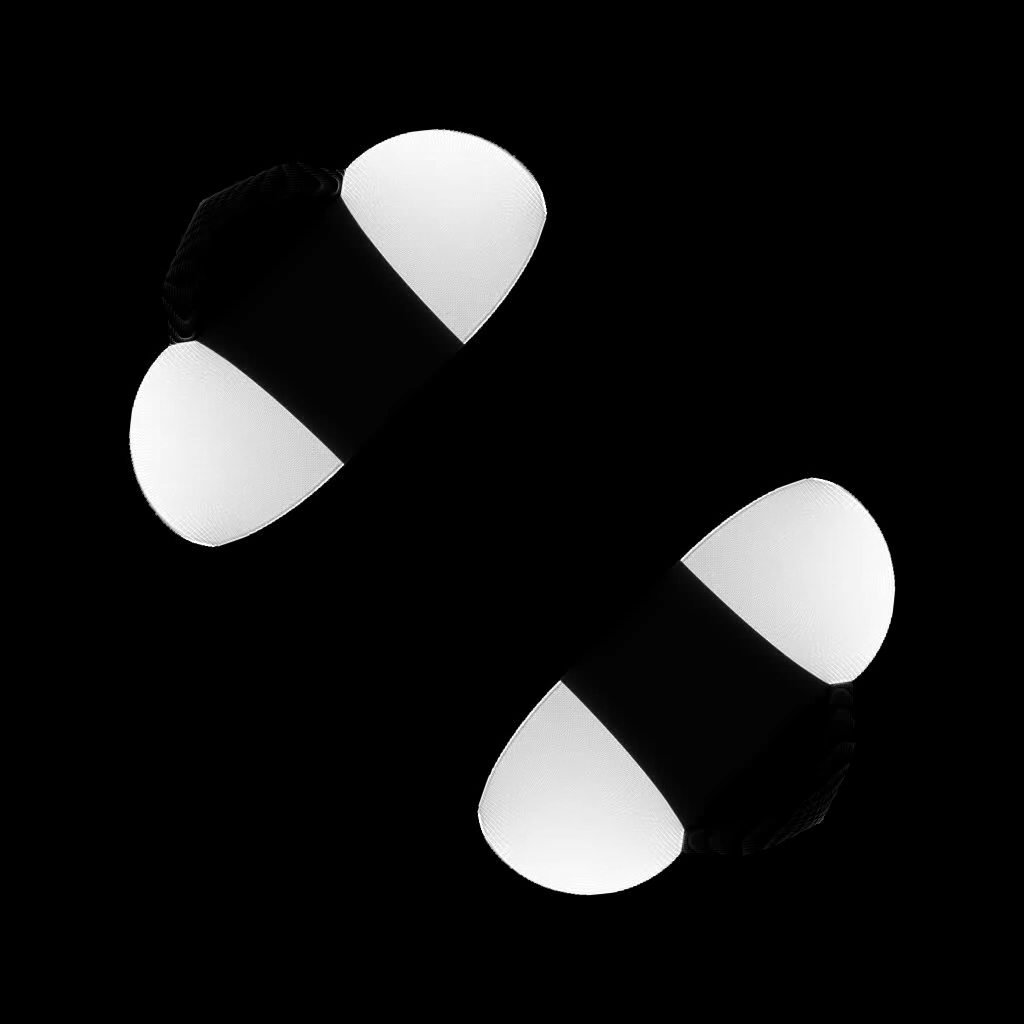}%
}
\\
\subfloat[$T=\frac{4}{5}$(Euclidean)]{
\includegraphics[width=0.17\textwidth]{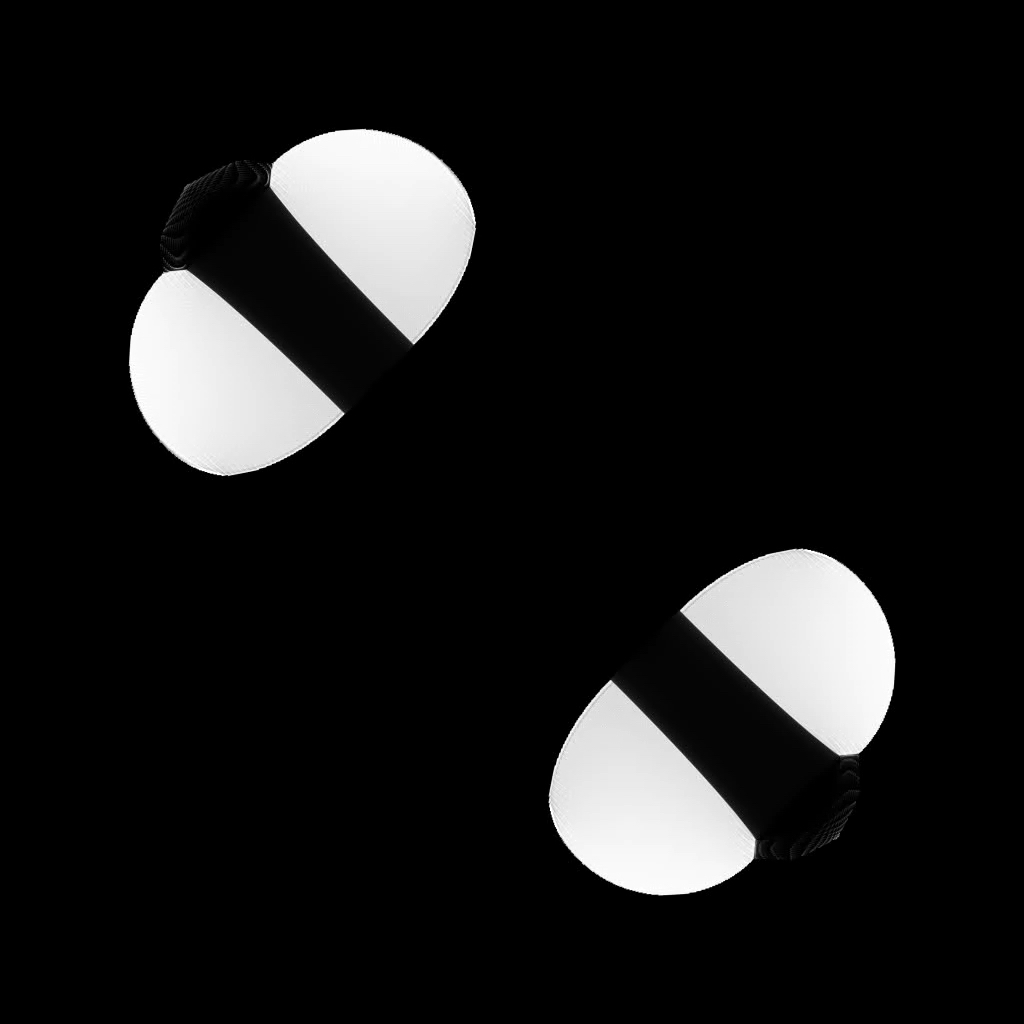}%
}
\subfloat[$T=\frac{4}{5}$(DI)]{
\includegraphics[width=0.17\textwidth]{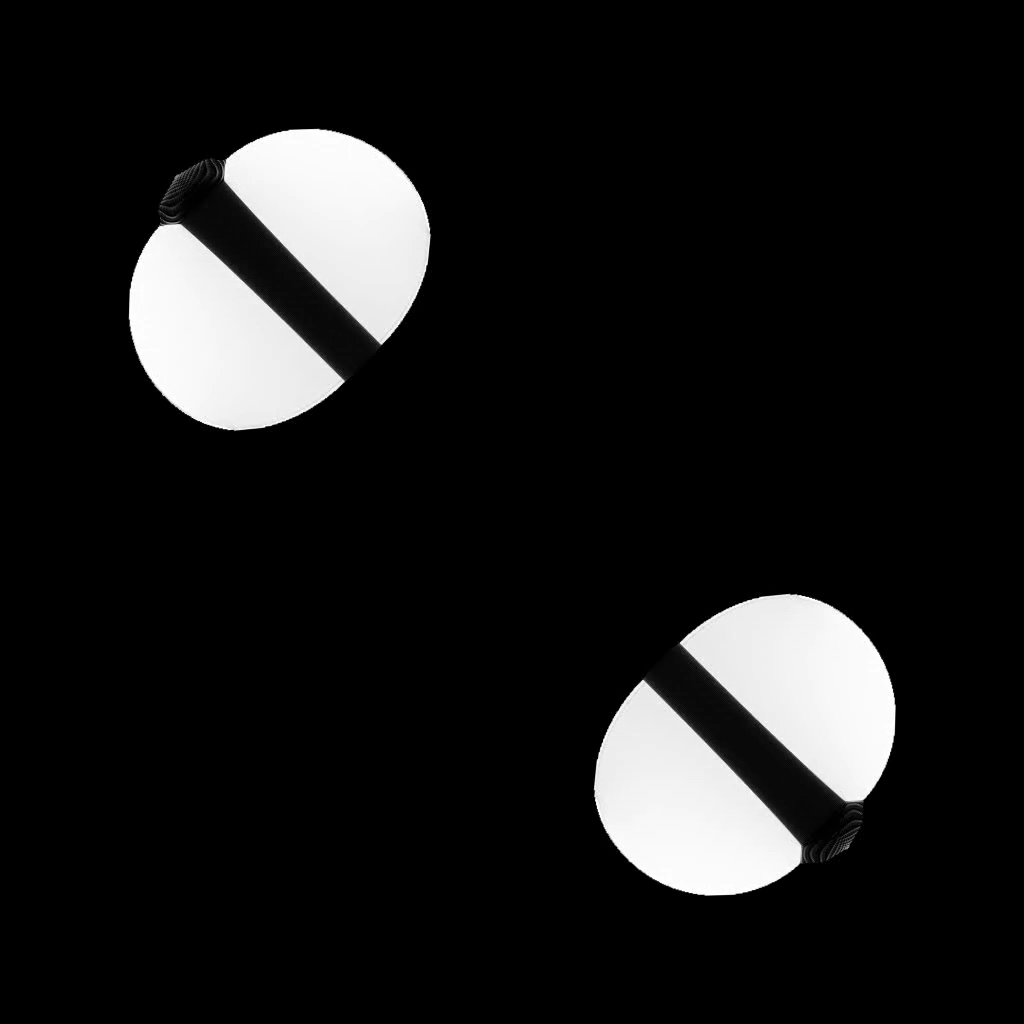}%
}
\caption{Evolution of measure under along optimal trajectories for Eucliean (Single Integrator) cost and Double Integrator cost}
\label{fig:OT}
\end{figure}

\textbf{Linear Quadratic costs} The optimal transport map can also be computed for the more general case of Linear quadratic costs given by,
\begin{eqnarray}
c(x,y) = \inf_{(\gamma,u) \in \Gamma \times \mathcal{U} } \int_0^1 \langle \gamma(t), Q \gamma(t) \rangle dt  + \int_0^1 \langle u(t), R u(t) \rangle dt \nonumber
\end{eqnarray}
subject to 
\begin{eqnarray}
& 
\dot{\gamma}
=A
\gamma(t)
+
B u(t) \\
&
\gamma(0)
 = \begin{bmatrix}
0 \\ \vdots \\ x
\end{bmatrix}
~~ 
\gamma(1)
= \begin{bmatrix}
0 \\  \vdots \\y
\end{bmatrix}
\end{eqnarray}
where $Q \in \mathbb{R}^{d \times d} $ is a positive semi-definite matrix such that $QP^{\perp}_E \mathbf{x} = Q \mathbf{x}$ for all $\mathbf{x} \in \mathbb{R}^d$, and $R \in \mathbb{R}^{d \times d}$ is a positive definite matrix.
In this case, it is known from \cite{hindawi2011mass}[Theorem 2.2] that there exist positive definite matrices $D,F \in \mathbb{R}^{d \times d}$ and a matrix $E \in \mathbb{R}^{d \times d}$ such that 
\[c(x,y) = \langle  \mathbf{x},D \mathbf{x} \rangle -\langle  \mathbf{x}, E \mathbf{y} \rangle +\langle  \mathbf{y}, F \mathbf{y} \rangle \]
where $\mathbf{x} = [0,...,x]^T$ and $\mathbf{y} = [0,...,y]^T$.
Once again, due to the special structure of $\mathbf{x}$ and $\mathbf{y}$, the above equality reduces to 
\[c(x,y) = D_{22}x^2 - E_{22}xy +F_{22}y^2 \]
Hence, the cost is just a quadratic form on $\mathbb{R}$, and one can use existing solvers such as the Back and forth method used for the minimum energy control of the double integrator, to solve a large class of LQ costs where the cost on $\mathbb{R}^d$ is of the form $\tilde{c}(x,y) = \sum_{i=1}^m c(x_i,y_i)$.

\section{Conclusion}
We established the well-posedeness of solutions to the 
optimal transport problem of linear time invariant systems over measures that are supported over the equilibrium set. Interesting future directions include using this idea to consider optimal transport of nonlinear systems such as those considered in \cite{agrachev2009optimal,elamvazhuthi2023dynamical}.                            

\bibliographystyle{plain}        
\bibliography{autosam}           

\begin{thebibliography}{10}

\bibitem{agrachev2009optimal}
Andrei Agrachev and Paul Lee.
\newblock Optimal transportation under nonholonomic constraints.
\newblock {\em Transactions of the American Mathematical Society},
  361(11):6019--6047, 2009.

\bibitem{balci2023covariance}
Isin~M Balci and Efstathios Bakolas.
\newblock Covariance steering of discrete-time linear systems with mixed
  multiplicative and additive noise.
\newblock In {\em 2023 American Control Conference (ACC)}, pages 2586--2591.
  IEEE, 2023.

\bibitem{biswal2021decentralized}
Shiba Biswal, Karthik Elamvazhuthi, and Spring Berman.
\newblock Decentralized control of multiagent systems using local density
  feedback.
\newblock {\em IEEE Transactions on Automatic Control}, 67(8):3920--3932, 2021.

\bibitem{chen2016relation}
Yongxin Chen, Tryphon~T Georgiou, and Michele Pavon.
\newblock On the relation between optimal transport and schr{\"o}dinger
  bridges: A stochastic control viewpoint.
\newblock {\em Journal of Optimization Theory and Applications}, 169:671--691,
  2016.

\bibitem{chen2016optimal}
Yongxin Chen, Tryphon~T Georgiou, and Michele Pavon.
\newblock Optimal transport over a linear dynamical system.
\newblock {\em IEEE Transactions on Automatic Control}, 62(5):2137--2152, 2016.

\bibitem{de2021discrete}
Mathias~Hudoba De~Badyn, Erik Miehling, Dylan Janak, Beh{\c{c}}et
  A{\c{c}}{\i}kme{\c{s}}e, Mehran Mesbahi, Tamer Ba{\c{s}}ar, John Lygeros, and
  Roy~S Smith.
\newblock Discrete-time linear-quadratic regulation via optimal transport.
\newblock In {\em 2021 60th IEEE Conference on Decision and Control (CDC)},
  pages 3060--3065. IEEE, 2021.

\bibitem{elamvazhuthi2018optimal}
Karthik Elamvazhuthi, Piyush Grover, and Spring Berman.
\newblock Optimal transport over deterministic discrete-time nonlinear systems
  using stochastic feedback laws.
\newblock {\em IEEE control systems letters}, 3(1):168--173, 2018.

\bibitem{elamvazhuthi2023dynamical}
Karthik Elamvazhuthi, Siting Liu, Wuchen Li, and Stanley Osher.
\newblock Dynamical optimal transport of nonlinear control-affine systems.
\newblock {\em Journal of Computational Dynamics}, pages 0--0, 2023.

\bibitem{fleig2017optimal}
Arthur Fleig and Roberto Guglielmi.
\newblock Optimal control of the fokker--planck equation with space-dependent
  controls.
\newblock {\em Journal of Optimization Theory and Applications}, 174:408--427,
  2017.

\bibitem{fleming2012deterministic}
Wendell~H Fleming and Raymond~W Rishel.
\newblock {\em Deterministic and stochastic optimal control}, volume~1.
\newblock Springer Science \& Business Media, 2012.

\bibitem{ito2023maximum}
Kaito Ito and Kenji Kashima.
\newblock Maximum entropy optimal density control of discrete-time linear
  systems and schr{\"o}dinger bridges.
\newblock {\em IEEE Transactions on Automatic Control}, 2023.

\bibitem{jacobs2020fast}
Matt Jacobs and Flavien L{\'e}ger.
\newblock A fast approach to optimal transport: The back-and-forth method.
\newblock {\em Numerische Mathematik}, 146(3):513--544, 2020.

\bibitem{kabir2021efficient}
Rabiul~Hasan Kabir and Kooktae Lee.
\newblock Efficient, decentralized, and collaborative multi-robot exploration
  using optimal transport theory.
\newblock In {\em 2021 American Control Conference (ACC)}, pages 4203--4208.
  IEEE, 2021.

\bibitem{karny1996towards}
Miroslav K{\'a}rn{\`y}.
\newblock Towards fully probabilistic control design.
\newblock {\em Automatica}, 32(12):1719--1722, 1996.

\bibitem{kolouri2017optimal}
Soheil Kolouri, Se~Rim Park, Matthew Thorpe, Dejan Slepcev, and Gustavo~K
  Rohde.
\newblock Optimal mass transport: Signal processing and machine-learning
  applications.
\newblock {\em IEEE signal processing magazine}, 34(4):43--59, 2017.

\bibitem{okamoto2018optimal}
Kazuhide Okamoto, Maxim Goldshtein, and Panagiotis Tsiotras.
\newblock Optimal covariance control for stochastic systems under chance
  constraints.
\newblock {\em IEEE Control Systems Letters}, 2(2):266--271, 2018.

\bibitem{pakniyat2022convex}
Ali Pakniyat.
\newblock A convex duality approach for assigning probability distributions to
  the state of nonlinear stochastic systems.
\newblock {\em IEEE Control Systems Letters}, 6:3080--3085, 2022.

\bibitem{hindawi2011mass}
J-B Pomet.
\newblock Mass transportation with lq cost functions.
\newblock {\em Acta applicandae mathematicae}, 113:215--229, 2011.

\bibitem{rubner2000earth}
Yossi Rubner, Carlo Tomasi, and Leonidas~J Guibas.
\newblock The earth mover's distance as a metric for image retrieval.
\newblock {\em International journal of computer vision}, 40:99--121, 2000.

\bibitem{santambrogio2015optimal}
Filippo Santambrogio.
\newblock Optimal transport for applied mathematicians.
\newblock {\em Birk{\"a}user, NY}, 55(58-63):94, 2015.

\bibitem{villani2009optimal}
C{\'e}dric Villani et~al.
\newblock {\em Optimal transport: old and new}, volume 338.
\newblock Springer, 2009.

\end{thebibliography}



\end{document}